\documentclass[leqno]{amsart}

\usepackage{amssymb,amsfonts,amsthm}
\usepackage{graphicx}
\usepackage{subfigure}
\usepackage{amsthm}

\usepackage{amsmath}
\usepackage{bm}             
\usepackage[mathscr]{eucal}
\usepackage{color}
\usepackage{cancel}
\usepackage{enumerate}
\usepackage{psfrag}
\usepackage{bbm}
\usepackage{appendix}

\newcommand{\bfR}{\mathbb{R}}

\newcommand{\mcA}{\mathcal{A}}
\newcommand{\mcM}{\mathcal{M}}
\newcommand{\M}{M_k}

\newcommand{\mcD}{\mathcal{D}}
\newcommand{\mcG}{\mathcal{G}}
\newcommand{\mcGx}{\mathcal{G}_x}
\newcommand{\mcGv}{\mathcal{G}_v}

\newcommand{\mcR}{\mathcal{R}}

\newcommand{\mcX}{\mathcal{X}}
\newcommand{\mcV}{\mathcal{V}}
\newcommand{\mcY}{\mathcal{Y}}

\newcommand{\mfV}{\mathfrak{V}}

\newcommand{\mcB}{\mathcal{B}}

\newcommand{\alt}[1]{#1}

\newtheorem{theorem}{Theorem}[section]
\newtheorem*{theorem*}{Theorem}
\newtheorem{corollary}[theorem]{Corollary}

\newtheorem{lemma}[theorem]{Lemma}

\theoremstyle{definition}

\newtheorem{remark}{Remark}[section]

\usepackage{fullpage}

\title{Scattering and Asymptotic Behavior of Solutions to the Vlasov-Poisson System in High Dimension}
\author{Stephen Pankavich}
\address{Department of Applied Mathematics and Statistics, Colorado School of Mines, Golden, CO 80401.}
\email{pankavic@mines.edu}

\date{\today}
\thanks{The author was supported in part by NSF grants DMS-1911145 and DMS-2107938.}

\begin{document}
\maketitle

\begin{abstract}
We consider the repulsive Vlasov-Poisson system in dimension $d \geq 4$.
A condition on the decay rate of the associated electric field is presented that guarantees the scattering and determination of the complete asymptotic behavior of large data solutions as $t \to \infty$.
More specifically, we show that under this condition the spatial average of the particle distribution function converges, and we establish the precise asymptotic profiles of the electric field and macroscopic densities.
An $L^\infty$ scattering result for the particle distribution function along the associated trajectories of free transport is also proved.
Finally, we construct small data solutions that display this asymptotic behavior. These solutions do not require smallness of $\|f_0\|_\infty$ or derivatives, as only a condition on integrated moments of the distribution function is imposed.
\end{abstract}

\section{Introduction}
We consider the electrostatic Vlasov-Poisson system with $t \geq 0$ and $x,v \in \mathbb{R}^d$, namely
\begin{equation}
\tag{VP}
\label{VP}
\left \{ \begin{aligned}
& \partial_{t}f+v\cdot\nabla_{x}f+ E \cdot\nabla_{v}f=0\\
& \rho(t,x)=   \int_{\mathbb{R}^d} f(t,x,v) \,dv\\
& E(t,x) = \nabla_x ( \Delta_x)^{-1} \rho(t,x) =\frac{1}{d\omega_d} \int_{\mathbb{R}^d} \frac{x-y}{\vert x - y \vert^d} \rho(t,y) \ dy
\end{aligned} \right .
\end{equation}
where $\omega_d$ represents the volume of the unit ball in $\mathbb{R}^d$.
Here, the particles are distributed in phase space at time $t \geq 0$ according to the function $f(t,x,v)$ 
and the initial distribution is given by $f(0,x,v) = f_0(x,v)$.
Additionally, $E(t,x)$ represents the electric field induced by the charged particles, $\rho(t,x)$ is the charge density,
and the current density is defined by
$$j(t,x) =   \int v f(t,x,v)  \ dv.$$
For simplicity, we have taken only a single species of charge and normalized the particle mass.
Assuming $f_0 \in L^1(\bfR^{2d})$, the solution remains integrable in phase space as the total charge is conserved in time, namely
$$ \iint f(t,x,v) \ dvdx =  \iint f_0(x,v) \ dvdx =: \mcM.$$
Furthermore, given smooth initial data \eqref{VP} has been shown to possess a smooth global-in-time solution \cite{LP, Pfaf, Schaeffer} for $d=3$, \alt{though such results have yet to be successfully extended to $d \geq 4$}. These global existence theorems depend upon either the propagation of higher (spatial, velocity, or transported) moments or precise estimates on the growth of the characteristics associated to \eqref{VP}, which are defined by
\begin{equation}
\label{char}
\left \{
\begin{aligned}
&\dot{\mcX}(t, \tau, x, v)=\mcV(t, \tau, x, v)\\
&\dot{\mcV}(t, \tau, x, v)=  E(t, \mcX(t, \tau, x, v))
\end{aligned}
\right.
\end{equation}
with initial conditions
$\mcX(\tau, \tau, x, v) = x$ and
$\mcV(\tau, \tau, x, v) = v.$
%
For additional background, we refer the reader to \cite{Glassey, Rein} as general references concerning \eqref{VP} and associated kinetic equations. 

Though the well-posedness of solutions to \eqref{VP} has been thoroughly studied, their time asymptotic behavior is less understood. Partial results concerning the asymptotic growth or decay of quantities in the system are known in some situations, including small data \cite{BD, Flynn, HRV, Pausader, Smulevici}, monocharged and spherically-symmetric data \cite{BCP1, Horst, Pankavich2020}, and lower-dimensional ($d=1,2$) settings \cite{BKR, BMP, GPS, GPS2, GPS4, Sch}.
In particular, these results provide either time asymptotic growth estimates of characteristics or decay estimates of the electric field and charge density.
We specifically note that small-data solutions for $d \geq 4$ were constructed in \cite{Smulevici} using vector field methods, and many of those constructed in three dimensions \cite{BD, Schaeffer21} can be extended to higher dimensions as well. While these solutions are shown to achieve sharp decay rates of the field and charge density, the asymptotic limits of such quantities are not determined. Our results remedy this issue and will apply immediately to those solutions, but we will also construct small data solutions with less restrictive smallness assumptions that further display the sharp large time behavior and asymptotic limits of the field, density, and the particle distribution. 

One generally expects that the field and charge density tend to zero as $t \to \infty$ like $t^{1-d}$ and $t^{-d}$, respectively, for all smooth solutions of \eqref{VP} due to the dispersive properties induced within the system by the transport operator $\partial_t + v \cdot \nabla_x$, the repulsive force generated by the electric field, and the velocity averaging inherent to these quantities. In fact, it is known that the Cauchy problem does not possess smooth steady states (cf., \cite{GPS5}).
%
That being said, it remains a longstanding open problem to demonstrate that for some $d \in \mathbb{N}$ \emph{all} smooth solutions of \eqref{VP} satisfy these decay properties or scatter to a profile along the trajectories generated by the (possibly modified) free transport operator as $t \to \infty$. 
Further evidence has been provided indicating that this behavior should be likely to occur in higher dimensions. Indeed, under the assumption of neutrality the well-known phenomenon of Landau Damping \cite{VM} has been shown to occur for any $d \in \mathbb{N}$, and even without neutrality, the dispersive effects within the Vlasov equation are expected to dominate the influence of the force field in higher dimensions more so than in lower dimensions \cite{BMP, Pankavich2021}.
Admittedly, the physical significance of \eqref{VP} becomes less obvious when posed in higher spatial and momentum dimensions, but understanding the inherent properties of the system and the behavior of solutions for $d \geq 4$ may lead to greater insights concerning the three-dimensional problem or its lower-dimensional analogues.
Hence, the goal of the current work is to establish a precise condition on the decay of the electric field that allows one to identify and establish the precise large-time behavior of any solution to \eqref{VP}, and also construct small data solutions that display exactly this asymptotic dynamic.

\subsection{Overview and Organization}

As we are primarily concerned with large time estimates, we use the notation
$$A(t) \lesssim B(t)$$
to represent the statement that there is $C > 0$
such that
$A(t) \leq CB(t)$
for all $t\geq 0$.
In this vein, we will often use the bracket notation
$$\langle u \rangle = \sqrt{1 + |u|^2}$$
for any $u \in \bfR^d$.
When necessary, $C$ will denote a positive constant (independent of the solution) that may depend upon dimension $d \geq 4$, $\alpha \geq 0$ (fixed below), and initial data and can change from line to line.
Throughout we take $f_0 \in C_0^1(\bfR^{2d})$, which represents the space of continuously differentiable functions that tend to zero as $|x| \to \infty$, so that we may consider smooth solutions, and let $f(t,x,v)$ denote the corresponding $C_0^1$ solution of \eqref{VP} launched by $f_0$.
Additionally, we take the initial distribution to be nonnegative, i.e. $f_0(x,v) \geq 0$, which is a property well known to be maintained in time by the solution.
Unlike \cite{Pankavich2021} we do not assume compactly-supported initial data, and instead propagate translated spatial moments of the distribution function in time.
In addition to the spatial decay as $|x| \to \infty$, the regularity assumptions on initial data could possibly be altered to arrive at similar convergence results in weaker topologies (see \cite{Pausader} for $d=3$).
Still, we will require only $C^1$ initial data rather than higher derivatives in $L^1$ or $L^2$ (as for vector field and harmonic methods) and can address all dimensions $d \geq 4$ simultaneously. 
%


\subsection{Main Results}
Fixing $\alpha \geq 0$, we define for any $t \geq 0$ the $k$th transported moment of $f(t)$ by
$$\M(t) = \iint |x - v(t+\alpha)|^k f(t,x,v) \ dv dx. $$
Our results can be summarized within three theorems.
First, we show that a sufficiently rapid rate of decay for the electric field implies the expected dispersive decay rate and uniformly bounded moments.
\begin{theorem}[Improved Decay]
\label{T0}
Let $d \geq 4$. Assume $M_n(0) < \infty$ for some $n > d(d-1)$ and $\alpha \geq 0$,
and 
\begin{equation}
\label{Assumption}
\tag{A}
\|E(t) \|_\infty \lesssim (t+\alpha)^{-a}
\end{equation}
for some $a > \tilde{a}(d) := \frac{d^2-d-4}{d^2-2d-2}$.
Then, 
$$\|E(t) \|_\infty \lesssim (t+\alpha)^{1-d} \qquad \mathrm{and} \qquad M_k(t) \lesssim 1$$
for all $0 \leq k \leq n$.
\end{theorem}
We note that $\tilde{a}(d)$ is decreasing for $d \geq 4$ with $\tilde{a}(d) \to 1$ as $d \to \infty$. 
Thus, a slower decay rate of the electric field is required in higher dimensions in order to obtain the dispersive decay rate stated in the conclusion of Theorem \ref{T0}.

\begin{remark}
As an alternative to \eqref{Assumption}, one may impose a growth condition on \alt{translated} moments, namely
$$M_n(t) \lesssim (t+\alpha)^{(2 - a)(n-2)}$$
for some $n > d(d-1)$ and $a > \tilde{a}(d)$,
in order to arrive at the same result.
\alt{For $d$ sufficiently large this yields 
$$M_n(t) \lesssim (t+\alpha)^{n-2 +\epsilon}$$
for $n > d(d-1)$ and some small $\epsilon > 0$.
The growth of such quantities has been investigated recently in \cite{Chen} for $d=3$. Though moments satisfying $n >d(d-1)=6$ were not included therein, the authors do obtain the estimate
$$ M_k(t) \lesssim (t+\alpha)^{\frac{9k-11}{7}}$$
for $k \to \frac{11}{2}^-$, which may be close to the above estimate, depending upon the value of $d$.}
In addition, instead of estimating moments, one can take $f_0 \in C^1_c(\bfR^{2d})$ and estimate the growth of the maximal translated spatial characteristics, namely
$$\mcR(t) = \sup \left  \{\left |\mcX(t,0,x,v) - (t + \alpha) \mcV(t,0,x,v) \right | : (x,v) \in \mathrm{supp}(f_0) \right \}.$$
Then, the field decay assumption \eqref{Assumption} can be replaced by a rate of growth on these characteristics, namely
$$\mcR(t) \lesssim (t+\alpha)^{1 - \frac{a}{d-1}}$$
for some $a  > \tilde{a}(d)$.
As $\tilde{a}(d) \to 1$ as $d \to \infty$, we have $1 - \frac{\tilde{a}(d)}{d-1} \to 1$ as $d \to \infty$, and using the best available estimate of the velocity support for $d = 3$ \cite{Yang}, one can currently derive the growth estimate
$$\mcR(t) \lesssim (t+\alpha)^\frac{9}{8}.$$ 
While this power would be insufficient to obtain \eqref{Assumption} even for $d \geq 4$, we note that it is not significantly distant from the required growth rate when $d$ is sufficiently large, and an advance in this direction may ultimately show that \alt{\emph{all}} smooth solutions of \eqref{VP} \alt{scatter and} satisfy the asymptotic behavior stated in Theorem \ref{T1}, especially in higher phase space dimensions.
\end{remark}

Next, we show that \eqref{Assumption} provides suitable information to obtain the precise asymptotic behavior of many quantities in the system, including the macroscopic densities and the distribution function along the flow generated by the linear transport operator. 
\begin{theorem}[Asymptotic Behavior]
\label{T1}
Let $d \geq 4$. Assume the conditions of Theorem \ref{T0} hold and $\langle x - \alpha v \rangle^p f_0 \in W^{1,\infty}(\mathbb{R}^{2d})$ for some $p > d+1$ and $\alpha \geq 0$.
Then, we have the following: 
\begin{enumerate}[(a)]
\item 
There exist a continuous $F_\infty \in L^1(\bfR^d) \cap L^\infty(\bfR^d)$ such that the spatial average
$$F(t,v) = \int f(t,x, v) \ dx$$
satisfies $F(t,v) \to F_\infty(v)$ uniformly as $t \to \infty$ with
$$\| F(t) - F_\infty \|_\infty \lesssim (t+\alpha)^{2-d}.$$

\item Define $E_\infty(v) = \nabla_v(\Delta_v)^{-1} F_\infty(v)$. Then, we have the self-similar asymptotic profiles
\begin{align*}
\sup_{x\in \bfR^d} \left | (t+\alpha)^{d-1} E(t,x) - E_\infty \left (\frac{x}{t+\alpha} \right ) \right | & \lesssim (t+\alpha)^{\frac{1-d}{d}},\\
\sup_{x \in \bfR^d}   \left | (t+\alpha)^d \rho(t,x) - F_\infty \left (\frac{x}{t+\alpha} \right) \right | & \lesssim (t+\alpha)^{-1},\\
\sup_{|x| \lesssim t + \alpha}  \left | (t+\alpha)^d j(t,x) - \frac{x}{t+\alpha}F_\infty\left (\frac{x}{t+\alpha} \right) \right | & \lesssim (t+\alpha)^{-1}.
\end{align*}

\item
There is a continuous $f_\infty \in L^1(\bfR^{2d}) \cap L^\infty(\bfR^{2d})$ such that
$$f(t,x +v(t+\alpha),v) \to f_\infty(x,v)$$
uniformly
as $t \to \infty$, namely we have the convergence estimate
$$\sup_{(x,v) \in \bfR^{2d}} \left | f(t,x +v(t+\alpha), v) - f_\infty(x,v) \right |  \lesssim (t+\alpha)^{2-d}.$$
\end{enumerate}
\end{theorem}

We note that due to the faster dispersive decay rate of the electric field for $d \geq 4$ in comparison with $d=3$, modifications to the trajectories along which the distribution function scatters are not needed, which differs from the results of \cite{Pausader, Pankavich2021}.

\begin{remark}
The reader may notice that the optimal rate of $(t+\alpha)^{-1}$ is not quite achieved for the convergence of the field to its limiting function. This can be remedied by further assuming $\langle x -\alpha v \rangle^p \langle v \rangle^q f_0 \in W^{1,\infty}(\bfR^{2d})$ 
for some $p > d+1$, $q > d, \alpha \geq 0 $ and uniformly bounding these moments in time. The methods of Lemma \ref{Funif} then allow one to show $$\|F(t) - F_\infty\|_1 \lesssim (t+\alpha)^{2-d},$$
which can be used in the proof of Lemma \ref{LField} to obtain the $(t+\alpha)^{-1}$ convergence rate of the field.
\end{remark}

\begin{remark}
Theorems \ref{T0} and \ref{T1} can be extended to $d=3$, but require a stronger decay assumption than \eqref{Assumption} and a modification to the trajectories along which $f$ scatters (see \cite{Pausader, Pankavich2021}).
Furthermore, the results of \cite{Pankavich2021} can be extended to $d \geq 4$, but require $f_0 \in C^1_c(\bfR^{2d})$.
For such initial data, the tools of \cite{Pankavich2021} show that no mass, momentum, or energy are lost in the limit, and the decay rates of the field and densities in the case of a neutral (i.e., $\mcM = 0$), multispecies system 
are actually faster than stated above if the limiting charge density vanishes.
\end{remark}

As solutions in higher dimensions have not been widely studied, with the exception of \cite{Smulevici}, only the small-data solutions \cite{BD, Schaeffer21} established in three-dimensions can be readily extended to $d \geq 4$ (see \cite{Ha1}).
For this reason, our last result serves to establish global-in-time solutions launched by small initial moments in $L^1(\bfR^{2d})$ by taking advantage of the increased dispersive effects of the system posed in $d \geq 4$ and their influence on the electric field.
One particular novelty of these solutions is that, unlike previous small data solutions \cite{BD, Pausader, Schaeffer21} for $d =3$ and \cite{Smulevici} for $d \geq 4$, they allow $\|f_0\|_\infty$ to be arbitrarily large and further do not require a smallness assumption on derivatives of initial data.

\begin{theorem}[Small Moment Solutions]
\label{T3}
Let $d \geq 4$. Assume $\langle x - \alpha v \rangle^p f_0 \in W^{1,\infty}(\mathbb{R}^{2d})$ for some $p > d+1$ and $\alpha > 0$. Then, there exists $\epsilon_0 > 0$, depending only upon $d$, $\alpha$, and $n$,  such that for all $0 < \epsilon \leq \epsilon_0$, if $M_n(0) \leq \epsilon$ for some $n > d(d-1)$ then the classical solution of \eqref{VP} launched by $f_0$ exists globally in time and satisfies
$$M_n(t) \leq 2\epsilon$$ 
and
$$\|E(t) \|_\infty \lesssim  \epsilon^\frac{(d-2)(d+1)}{d(n-2)} (t+\alpha)^{1-d}.$$
Furthermore, the conclusions of Theorem \ref{T1} immediately apply.
\end{theorem}

\begin{remark}
The introduction of the parameter $\alpha > 0$ is only used to unify the three results and eliminate the singularity within the estimates that occurs as $t \to 0^+$. One may take $\alpha = 0$ in Theorems \ref{T0} and \ref{T1} and use known estimates (c.f. \cite{Castella}) to control all quantities for sufficiently small time. For this reason, we will inherently assume $\alpha > 0$ within the proofs of these two theorems. However, as Theorem \ref{T3} focuses on small data solutions, $\alpha > 0$ is needed to obtain control of the field near $t = 0$.
Furthermore, note that taking $\alpha = 0$ in Theorem \ref{T0} requires only $|x|^n f_0 \in L^1(\bfR^{2d})$, while the additional condition in Theorem \ref{T1} is merely  $ \langle x \rangle^p f_0 \in W^{1,\infty}(\bfR^{2d})$.
Hence, our results allow for classical solutions with heavy tails in the velocity variable, as well as infinite kinetic energy.
\end{remark}

\begin{remark}
The reader will note that we only study the \emph{repulsive} Vlasov-Poisson system, rather than also considering the possibility of an \emph{attractive} force field arising, for instance, within the analogous gravitational model. Indeed, while our methods do depend upon the repulsive nature of the force field (see the comment preceding Lemma \ref{mompot}), this dependence may not be crucial to our main arguments. Hence, it is likely that similar methods can be used in the attractive case to study the asymptotic behavior of small data solutions (cf., \cite{Pausader}). However, it is known \cite{HorstBlowup} that any smooth solution of the attractive Vlasov-Poisson system for $d \geq 4$ that possesses negative energy can only exist on a finite time interval. Thus, the dynamical behavior of solutions in the attractive case may be quite different.
\end{remark}

\subsection{Strategy of the Proofs}
To establish the theorems we will reformulate the original problem within a dispersive reference frame that is co-moving with the particles.
More specifically, let 
$$g(t,x,v) = f(t, x+v(t+\alpha), v)$$
and apply a change of variables inspired by \cite{Pausader} (see the proofs of Lemmas \ref{LField} and \ref{LDensity}) to the field and charge density so that \eqref{VP} becomes
\begin{equation}
\tag{VP$_g$}
\label{VPg}
\left \{ \begin{aligned}
& \partial_{t}g -  (t+\alpha) E(t,x+v(t+\alpha))\cdot \nabla_{x}g+  E(t,x+v(t+\alpha)) \cdot\nabla_{v}g=0\\
& E(t,x+v(t+\alpha)) = (t+\alpha)^{1-d}   \frac{1}{d\omega_d} \iint \frac{\xi}{|\xi|^d} \ g \left (t, w, v - \xi + \frac{x-w}{t+\alpha} \right ) \ dwd\xi
\end{aligned} \right .
\end{equation}
with
$$\rho(t,x)= (t+\alpha)^{-d}   \int_{\mathbb{R}^d} g \left (t,w,\frac{x-w}{t+\alpha} \right)\,dw$$
and the initial conditions $g(0,x,v) = f_0(x+\alpha v,v)$.
Additionally, the translated moments merely become
$$M_k(t) = \iint |x|^k g(t,x,v) \ dv dx $$
for every $k \geq 0$.
This reformulation is performed because $g$ possesses nicer properties than the original distribution function $f$. Indeed, both spatial moments and derivatives of $g$ can be uniformly bounded in time (Lemmas \ref{momrho} and \ref{DEf}), while the corresponding quantities for $f$ must grow in time.
Additionally, we note that the convolution in the electric field is now in the velocity variable rather than the spatial variable. 
Hence, as $t \to \infty$ one expects 
$$E(t,x+v(t+\alpha)) \sim (t+\alpha)^{1-d}  \frac{1}{d\omega_d} \iint \frac{\xi}{|\xi|^d} g(t, w, v-\xi) \ dw d\xi = (t+\alpha)^{1-d}\nabla_v \left (\Delta_v \right )^{-1} F(t,v)$$
locally in $x$.
Because of this, estimates of the field require an understanding of the growth of spatial moments of $g$ to control $F(t,v)$, and velocity derivatives of $g$ will be instrumental to demonstrating the asymptotic limit of the field.
Thus, our results may also provide better tools to obtain \emph{a priori} estimates on the growth of moments of $g$ and velocity derivatives $\nabla_v g$, which are the two main ingredients in the theorems.

In the next section, we establish preliminary estimates on the electric field and integrated moments of the distribution function, then use them to prove Theorem \ref{T0}. 
Section \ref{sect:T2} assumes the decay rate of the field guaranteed by Theorem \ref{T0} and then establishes estimates on derivatives of the field and the convergence of the spatial average.
The precise asymptotic behavior of the electric field and the charge and current densities is also obtained, as is the scattering of the distribution function stated in Theorem \ref{T1}. Finally, the construction of global-in-time small data solutions via the proof of Theorem \ref{T3} is provided in Section \ref{sect:T3}.

\section{Preliminary Lemmas \& Proof of Theorem \ref{T0}}
\label{sect:T1}
\label{Lemmas}

We first generalize an identity described within \cite{IR, Perthame} for the three-dimensional Vlasov-Poisson system and use it to obtain \emph{a priori} bounds on the second moment of the translated distribution function, as well as a decay estimate for the potential energy.

\begin{lemma}
\label{IR}
Let $f (t) \in C^1_0(\bfR^{2d}) $ be a classical solution of \eqref{VP} and $M_2(0) < \infty$ for some $\alpha \geq 0$. Then, the following identity holds
$$\frac{d}{dt} \left ( M_2(t) + (t+\alpha)^2\int |E(t,x)|^2 \ dx \right ) = (4-d) (t+\alpha)\int |E(t,x)|^2 \ dx.$$
\end{lemma}
\begin{proof}
We first recall that the potential $U(t,x)$ satisfies $\Delta U = \rho$, and thus
$$U(t,x) = \frac{1}{d(2-d)\omega_d} |x|^{2-d} \star \rho(t,x).$$
With this, the field is given by
$$E(t,x) = \nabla_x U(t,x) = \frac{1}{d\omega_d} \frac{x}{|x|^d}\star \rho(t,x).$$

Then, computing the time derivative of the second transported moment gives
\begin{eqnarray*}
M_2'(t) & = & -2(t+\alpha)  \iint (x-v(t+\alpha)) \cdot  E(t,x) f(t,x, v) \ dv dx\\
& = & -2(t+\alpha) \int x \cdot E(t,x) \rho(t,x) \ dx + 2(t+\alpha)^2 \int E(t,x) \cdot j(t,x) \ dx\\
& =: & -2(t+\alpha) \mcA(t) + 2(t+\alpha)^2 \mcB(t).
\end{eqnarray*}
Further, we compute
\begin{eqnarray*}
\mcA(t) & = & \int x \cdot E(t,x) \rho(t,x) \ dx \\
& = &  \frac{1}{d\omega_d} \iint x \cdot \frac{x-y}{|x-y|^d} \rho(t,x) \rho(t,y) \ dy \ dx\\ 
& = & \frac{1}{d\omega_d} \iint |x-y|^{2-d} \rho(t,x) \rho(t,y) \ dy \ dx - \int y \cdot E(t,y) \rho(t,y) \ dy\\
& = & (2-d) \int U(t,x) \rho(t,x) \ dx -\mcA(t).
\end{eqnarray*}
Hence, upon using $\rho = \Delta U$ and integrating by parts we find
$$\mcA(t) = \frac{2-d}{2}\int U(t,x) \rho(t,x) \ dx = \frac{d-2}{2}\int |E(t,x)|^2 \ dx.$$

Next, we compute the $\mcB$ term using the continuity equation
$$\partial_t \rho + \nabla_x \cdot j = 0,$$ which is obtained by integrating the Vlasov equation in $v$. 
Then, integrating by parts we find
\begin{eqnarray*}
\mcB(t) & = & \int E(t,x) \cdot j(t,x) \ dx\\
& = &- \int U(t,x) \nabla_x \cdot j(t,x) \ dx\\
& = & \int U(t,x) \partial_t \rho(t,x) \ dx\\
& = & \frac{1}{2} \frac{d}{dt} \int U(t,x) \rho(t,x) \ dx\\
& = & - \frac{1}{2} \frac{d}{dt} \int |E(t,x)|^2 \ dx.
\end{eqnarray*}

With these expressions, the derivative of the transported second moment becomes
$$M_2'(t)  = (2-d)(t+\alpha)\int |E(t,x)|^2 \ dx - (t+\alpha)^2 \frac{d}{dt} \int |E(t,x)|^2 \ dx.$$
The right side can be rewritten as
$$- \frac{d}{dt} \left ( (t+\alpha)^2 \int |E(t,x)|^2 \ dx \right ) + (4-d)(t+\alpha)\int |E(t,x)|^2 \ dx.$$
and the identity follows.
\end{proof}

As mentioned in \cite{IR}, this identity holds only for the \emph{repulsive} Vlasov-Poisson system, as it implies decay of the potential energy, which we now demonstrate. 

\begin{lemma}
\label{mompot}
For $d \geq 4$ we have
$$\| E(t) \|_2 \lesssim (t+\alpha)^{-1}$$
and
$$ M_2(t) \lesssim 1.$$
\end{lemma}

\begin{proof}
We let 
$$\psi(t) = (t+\alpha)^2 \|E(t) \|_2^2$$
so that the identity in Lemma \ref{IR} reads
$$\frac{d}{dt} \biggl ( M_2(t) + \psi(t) \biggr ) = (4-d)\frac{\psi(t)}{t+\alpha}$$
%
With this, we have
$$\frac{d}{dt} \left ( M_2(t) + \psi(t) \right ) \leq 0$$
as $d \geq 4$ and $\psi(t) \geq 0$.
Of course, this implies
$$\psi(t) \leq M_2(0) + \psi(0) - M_2(t) \lesssim 1$$
and
$$M_2(t) \leq M_2(0) + \psi(0) - \psi(t) \lesssim 1,$$ 
and the stated estimates follow.
\end{proof}

Next, we obtain improved field decay rates that follow from the main assumption \eqref{Assumption}.
First, we state a standard estimate on the gradient of the inverse Laplace operator, which will be used throughout.
%

\begin{lemma}[]
\label{LE}
For any $1 \leq p < d < q \leq \infty$ and $\phi \in L^p(\bfR^d) \cap L^q(\bfR^d)$, we have
$$\|\nabla (\Delta)^{-1} \phi\|_\infty \lesssim \| \phi \|_p^\frac{p(q-d)}{d(q-p)} \| \phi \|_q^\frac{q(d-p)}{d(q-p)}.$$
In particular, for $d \neq 1$, choosing $p=1$ and $q = \infty$ yields
$$\|\nabla (\Delta)^{-1} \phi\|_\infty \lesssim \| \phi \|_1^\frac{1}{d} \| \phi \|_\infty^\frac{d-1}{d}.$$
Similarly, for any $k > d(d-1)$ with $d \geq 4$ we may choose $p = \frac{d+2}{d}$ and $q = \frac{d+k}{d}$ to find
$$\| E(t) \|_\infty \lesssim \| \rho(t) \|_{\frac{d+2}{d}}^{\frac{(d+2)(k-d^2+d)}{d^2(k-2)}} \ \| \rho(t) \|_{\frac{d+k}{d}}^{\frac{(d+k)(d-2)(d+1)}{d^2(k-2)}}.$$
\end{lemma}
\begin{proof}
To establish the estimates, we decompose the spatial integral into contributions near and far from the singularity, so that
$$|\nabla (\Delta)^{-1} \phi(x)| \lesssim \int_{|x-y| < R} \frac{\phi(y)}{|x-y|^{d-1}} \ dy + \int_{|x-y| > R} \frac{\phi(y)}{|x-y|^{d-1}} \ dy := I + II.$$
As $q > d$ the first portion provides the estimate
$$I \lesssim \|\phi(t) \|_q \left ( \int_0^R r^{(d-1)\left ( 1 - \frac{q}{q-1} \right )} \ dr \right )^{\frac{q-1}{q}} \lesssim \|\phi(t) \|_q R^{\frac{q-d}{q}}, $$
while the second analogously yields
$$II \lesssim \| \phi(t) \|_p R^{\frac{p-d}{p}} $$
as $p< d$.
Combining these estimates and choosing
$$R = \left ( \frac{\|\phi(t) \|_p}{\|\phi(t) \|_q} \right )^\frac{pq}{d(q-p)}$$
gives 
$$|\nabla (\Delta)^{-1} \phi(x)| \lesssim \| \phi \|_p^\frac{p(q-d)}{d(q-p)} \| \phi \|_q^\frac{q(d-p)}{d(q-p)},$$
and the stated estimates follow.
\end{proof}

\begin{lemma}
\label{Lrho}
For any $k \geq 0$ we have
$$ \| \rho(t) \|_{\frac{d+k}{d}} \lesssim (t+\alpha)^\frac{-kd}{d+k} M_k(t)^\frac{d}{d+k}.$$
In particular, for $d \geq 4$ using Lemma \ref{mompot} with $k=2$ gives
$$ \| \rho(t) \|_{\frac{d+2}{d}} \lesssim (t+\alpha)^\frac{-2d}{d+2}.$$
\end{lemma}

\begin{proof}
For any $R > 0$, we decompose the integral into
\begin{eqnarray*}
|\rho(t,x)| & \lesssim &  \int_{\left | \frac{x}{t+\alpha} - v \right | < \frac{R}{t+\alpha}} f(t,x,v) \ dv +  \int_{\left | \frac{x}{t+\alpha} - v \right | > \frac{R}{t+\alpha}} f(t,x,v) \ dv\\
& \lesssim & \left ( \frac{R}{t+\alpha} \right )^d + R^{-k} \int |x - v(t+\alpha)|^k f(t,x,v) \ dv\\
& \lesssim & R^d (t+\alpha)^{-d} + R^{-k}m_k(t,x)
\end{eqnarray*}
where we have denoted $m_k(t,x) =  \int |x - v(t+\alpha)|^2 f(t,x,v) \ dv$. 
Choosing
$$R = (t+\alpha)^\frac{d}{d+k} m_k(t,x)^\frac{1}{d+k}$$
yields
$$\rho(t,x)  \lesssim (t+\alpha)^{-\frac{kd}{d+k}} m_k(t,x)^\frac{d}{d+k}$$
and thus
$$\int \rho(t,x)^\frac{d+k}{d} \ dx \lesssim (t+\alpha)^{-k} \int m_k(t,x) \ dx = (t+\alpha)^{-k}  M_k(t).$$
Raising this inequality to the $\frac{d}{d+k}$ power gives the first result, and invoking
Lemma \ref{mompot} for $k=2$ produces the latter estimate.
\end{proof}

\begin{corollary}
\label{EM}
Combining the final estimate of Lemma \ref{LE} and the results of Lemma \ref{Lrho} for any $k > d(d-1)$ provides the estimate
$$\| E(t) \|_\infty \lesssim (t+\alpha)^{1-d} \ M_k(t)^\frac{(d-2)(d+1)}{d(k-2)}.$$
\end{corollary}

Now that we have established control of the field in terms of moments, we will bound moments in terms of the supremum of the field.
This will be accomplished by propagating moments in time via the Vlasov equation, but first we need an interpolation estimate for $M_k(t)$.

\begin{lemma}
\label{LM}
For any $\ell \geq 0$, $p \in [0,\ell]$ and $q \in[0,\infty)$, we have
$$M_\ell(t) \leq M_{\ell-p}(t)^\frac{q}{p+q}M_{\ell+q}(t)^\frac{p}{p+q}.$$
\end{lemma}

\begin{proof}
The proof is straightforward, but we include it for completeness. Separating the estimates into regions within which the moments are small and large, respectively, we find
\begin{eqnarray*}
M_\ell(t) & = & \iint_{|x - v(t+\alpha)| < R} |x - v(t+\alpha)|^{\ell} f(t,x,v) \ dvdx +  \iint_{|x - v(t+\alpha)| > R} |x - v\alt{(t+\alpha)}|^{\ell} f(t,x,v) \ dvdx\\
& \leq & R^{p} M_{\ell-p}(t) +  R^{-q} M_{\ell+q}(t)
\end{eqnarray*}
for any $\ell \geq p \geq 0$ and $q\geq 0$.
Optimizing in $R$ yields $R = \left (\frac{M_{\ell+q}(t)}{M_{\ell-p}(t)} \right )^\frac{1}{p+q}$,
and the estimate follows with this choice of $R$.
\end{proof}

\begin{lemma}
\label{ME}
For any $\alt{k \geq 3}$, if $M_k(0) < \infty$ then for all $t \geq 0$
$$M_k(t)^\frac{1}{k-2} \leq M_k(0)^\frac{1}{k-2} + C\int_0^t (s+ \alpha) \| E(s) \|_\infty \ ds.$$
\end{lemma}
\begin{proof}
Taking a derivative of $M_k(t)$ gives
\begin{eqnarray*}
|M_k'(t) | & \lesssim & (t+\alpha) \left | \iint |x - v(t+\alpha)|^{k-2} (x -v(t+\alpha)) \cdot E(t,x) f(t,x,v) \ dv dx \right |\\
& \lesssim & (t+ \alpha) \| E(t) \|_\infty M_{k-1}(t).
\end{eqnarray*}
Using Lemma \ref{LM} for any $\ell \geq 2$ with $p=\ell-2$ and $q = 1$ yields
$$M_\ell(t) \leq CM_{\ell+1}(t)^\frac{\ell-2}{\ell-1}$$
as $M_2(t) \lesssim 1$ due to Lemma \ref{mompot}.
Then, taking $\ell=k-1$ gives
$$M_{k-1}(t) \lesssim M_k(t)^\frac{k-3}{k-2}.$$
Using this in the above inequality for the derivative then implies
$$|M_k'(t) |  \lesssim (t+ \alpha) \| E(t) \|_\infty M_k(t)^\frac{k-3}{k-2},$$
and thus
$$\left |\frac{d}{dt} \left ( M_k(t)^\frac{1}{k-2} \right ) \right |  \lesssim (t+ \alpha) \| E(t) \|_\infty. $$
Integrating yields the stated result, namely
$$M_k(t)^\frac{1}{k-2} \leq M_k(0)^\frac{1}{k-2} + C\int_0^t (s+ \alpha) \| E(s) \|_\infty \ ds.$$
\end{proof}

With these estimates established, we can now prove the first theorem.
\begin{proof}[Proof of Theorem \ref{T0}]
Assuming \eqref{Assumption} for some $a > \tilde{a}(d) := \frac{d^2-d-4}{d^2-2d-2}$,
we let $\epsilon = a - \tilde{a}(d) > 0$ so that
$$\|E(t) \|_\infty \lesssim (t+\alpha)^{-\tilde{a}(d) - \epsilon}.$$
Using Lemma \ref{ME} with $k=n$ and inserting the above field estimate gives
$$M_n(t)^\frac{1}{n-2} \lesssim M_n(0)^\frac{1}{n-2} + \int_0^t (s+ \alpha)^{1-\tilde{a}(d) - \epsilon} \ ds \lesssim M_n(0)^\frac{1}{n-2} + \max \left \{ 1, (t+ \alpha)^{2-\tilde{a}(d) - \epsilon} \right \} $$
and thus
$$M_n(t) \lesssim \max \{ 1, (t+\alpha)^r  \}$$
where
$$r = \left (2-\tilde{a}(d) - \epsilon \right )(n-2).$$
Next, we use Corollary \ref{EM} so that
$$\| E(t) \|_\infty \lesssim (t+\alpha)^{1-d}  \max \{ 1, (t+\alpha)^r  \}^\frac{(d-2)(d+1)}{d(n-2)} \lesssim  \max \{ (t+\alpha)^{1-d}, (t+\alpha)^s  \}$$
where
$$s = 1-d + r\frac{(d-2)(d+1)}{d(n-2)} = 1- d + \frac{ \left (1 - \epsilon-( \tilde{a}(d) - 1) \right ) (d-2)(d+1)}{d}.$$
Using the identity
$$(\tilde{a}(d) - 1)(d-2)(d+1) = d \tilde{a}(d) - 2,$$
a brief calculation shows that this exponent can be rewritten as
$$s = -\tilde{a}(d) - \epsilon \frac{(d-2)(d+1)}{d}.$$
Note that the original assumption on the decay of the field, namely \eqref{Assumption}, can be expressed as 
$$\|E(t) \|_\infty \lesssim \max \left \{ (t+\alpha)^{1-d}, (t+\alpha)^{-\tilde{a}(d) - \epsilon} \right \}.$$
Thus, we have achieved an improved estimate, given by
$$\|E(t) \|_\infty \lesssim \max \left \{ (t+\alpha)^{1-d}, (t+\alpha)^{-\tilde{a}(d) - \epsilon\frac{(d-2)(d+1)}{d}} \right \}$$
as
\begin{equation}
\label{ineq}
\frac{(d-2)(d+1)}{d} = \left ( 1- \frac{2}{d} \right )(d+1) \geq \frac{5}{2}
\end{equation}
for $d \geq 4$.
Iterating this process then gives
$$\|E(t) \|_\infty \lesssim \max \left \{ (t+\alpha)^{1-d}, (t+\alpha)^{-\tilde{a}(d) - \epsilon \left (\frac{(d-2)(d+1)}{d}\right )^k} \right \}$$
for any $k \in \mathbb{N}$.
Due to \eqref{ineq}, taking $k$ sufficiently large implies
$$(t+\alpha)^{-\tilde{a}(d) - \epsilon \left (\frac{(d-2)(d+1)}{d}\right )^k} \lesssim (t+\alpha)^{1-d},$$
and thus the sharp decay rate for the field is ultimately achieved, namely
$$\|E(t) \|_\infty \lesssim (t+\alpha)^{1-d}.$$
Lemma \ref{ME} then provides the moment bound, namely
$$M_n(t) \lesssim \left ( M_n(0)^\frac{1}{n-2} + \int_0^t (s+ \alpha) \| E(s) \|_\infty \ ds \right )^{n-2} \lesssim \left ( 1+ \int_0^t (s+\alpha)^{2-d} \ ds \right )^{n-2} \lesssim 1.$$
Finally, the bound on $M_k(t)$ for $0 \leq k \leq n$ is achieved via interpolation with $M_0(t) = \mcM$ and $M_n(t)$.
\end{proof}

\section{Asymptotic Behavior and Proof of Theorem \ref{T1}}
\label{sect:T2}

Next, we establish a number of lemmas that will culminate in the proof of Theorem \ref{T1}. 
In view of Theorem \ref{T0}, we assume throughout this section that the electric field decays at the rate stated in the conclusion of that result, namely
$$\|E(t) \|_\infty \lesssim (t+\alpha)^{1-d}$$
with  uniform bounds on moments
$$M_k(t) \lesssim 1$$
for all $0 \leq k \leq n$.
%

Prior to stating the lemmas, we first introduce some notation relating to the translated distribution function.
As mentioned in the introduction, we let
$$g(t,x,v) = f(t,x+v(t+\alpha), v).$$
From the original characteristics given by \eqref{char}, we define the new spatial characteristics associated to this distribution function by
$$\mcY(t, \tau, x, v) = \mcX(t, \tau, x, v) - (t+\alpha) \mcV(t, \tau, x, v)$$ 
with
$\mcY(\tau, \tau, x, v) = x - (\tau +\alpha)v$.

%
As our approach relies heavily upon the growth of the spatial moments and velocity derivatives of $g$, 
we further define the useful quantities
$$\mcG(t) 
= 1 + \sup_{x,v \in \bfR^d} \biggl ( \langle x \rangle^p g(t,x,v) \biggr ) ,$$
$$\mcGx(t) 
= 1 + \sup_{x,v \in \bfR^d} | \langle x \rangle^p \nabla_x g(t,x,v) | ,$$
and
$$\mcGv(t) 
= 1 + \sup_{x,v \in \bfR^d} | \langle x \rangle^p \nabla_v g(t,x,v) | .$$
Notice that
$$\mcG(0) = 1 + \sup_{x,v \in \bfR^d} | \langle x \rangle^p f_0(x+\alpha v,v) | = 1 + \sup_{x,v \in \bfR^d} | \langle x- \alpha v \rangle^p f_0(x,v) |$$
and similarly for $\mcGx(0)$ and $\mcGv(0)$; hence, these quantities are all initially finite due to the assumptions of the theorem.

Our first lemma uses the field decay to uniformly bound the moments of $g$ and obtain the sharp decay rate of the charge density.
\begin{lemma}
\label{momrho}
We have
$$\mcG(t) \lesssim 1 \qquad \mathrm{and} \qquad \| \rho(t) \|_\infty \lesssim (t+ \alpha)^{-d}.$$
\end{lemma}
\begin{proof}
Define the operator $\mfV$ by
$$\mfV h = \partial_t h - (t+\alpha) E(t, x+ v(t+\alpha)) \cdot \nabla_x h + E(t, x+ v(t+\alpha)) \cdot \nabla_v h $$
for any $h = h(t,x,v)$ so that $\mfV g = 0$.
Then, applying the operator to $\langle x \rangle^p g$ yields
$$\mfV \left (\langle x \rangle^p g \right ) = - p (t+\alpha) \langle x \rangle^{p-2}  x \cdot E(t, x+ v(t+\alpha)) g(t,x,v).$$
Inverting the operator via integration along characteristics then gives
$$\langle x \rangle^p g(t,x,v) = \langle \mcY(0) \rangle^p g(0,\mcY(0), \mcV(0)) - p \int_0^t (s+\alpha) \langle \mcY(s) \rangle^{p-2}  \mcY(s) \cdot E(s, \mcY(s))  g(s, \mcY(s), \mcV(s)) \ ds,  $$
and this further yields
\begin{eqnarray*}
\| \langle x \rangle^p g(t) \|_\infty & \leq &  \| \langle x \rangle^p g(0) \|_\infty + C\int_0^t (s + \alpha) \|E(s) \|_\infty \| \langle x \rangle^p g(s) \|_\infty \ ds\\
 & \leq &  \| \langle x - \alpha v \rangle^p f_0 \|_\infty + C\int_0^t (s + \alpha)^{2-d} \| \langle x \rangle^p g(s) \|_\infty \ ds.
\end{eqnarray*}
Applying Gronwall's inequality, we find
$$\| \langle x \rangle^p g(t) \|_\infty \lesssim \exp \left ( \int_0^t (s + \alpha)^{2-d} \ ds \right ) \lesssim 1,$$
which gives the former result.

With this, we estimate the charge density using the change of variables $w = x -v(t + \alpha)$ so that
\begin{eqnarray*}
\rho(t,x)  & = & \int g(t,x-v(t+\alpha), v) \ dv\\
& = & (t + \alpha)^{-d} \int g \left (t,w, \frac{x-w}{t+\alpha} \right ) \ dw\\
& \leq & (t + \alpha)^{-d} \int \langle w \rangle^{-p} \sup_{\xi \in \bfR^d}\biggl ( \langle w \rangle^p g \left (t,w, \xi \right ) \biggr ) \ dw\\
& \leq & (t + \alpha)^{-d}  \mcG(t)  \left ( \int \langle w \rangle^{-p} \ dw\right ) \\
& \lesssim & (t + \alpha)^{-d}, 
\end{eqnarray*}
which provides the latter result.
\end{proof}

The decay of the field and charge density lead directly to estimates of field derivatives and derivatives of the distribution function.
In particular, we show that derivatives of $g$ are uniformly bounded.
\begin{lemma}
\label{DEf}
We have the estimates 
$$\|\nabla_x E(t) \|_\infty \lesssim (t+\alpha)^{-d}\ln(1+t+\alpha) , \qquad \mcGv(t) \lesssim 1, \qquad \mathrm{and} \qquad \mcGx(t) \lesssim 1.$$
\end{lemma}

\begin{proof}
We will establish an extension of the well-known three-dimensional estimate of field derivatives (c.f. \cite[p. 122-123]{Glassey}) to higher dimensions.
In particular, we apply a derivative to the field, use the identity
$$\partial_{x_j} \left ( \frac{x_i - y_i}{|x - y|^d} \right ) = - \partial_{y_j} \left ( \frac{x_i - y_i}{|x - y|^d} \right ),$$
and integrate by parts to find for any $0 < R_1 < R_2$
\begin{eqnarray*}
\partial_{x_j} E_i(t,x) & = & C \int_{\mathbb{R}^d} \partial_{x_j} \left ( \frac{x_i - y_i}{|x - y|^d} \right ) \rho(t,y) \ dy\\
& = & C \int_{\mathbb{R}^d} \frac{x_i - y_i}{|x - y|^d} \partial_{y_j}  \rho(t,y) \ dy\\
& = & C \int_{|x-y| < R_1} \frac{x_i - y_i}{|x - y|^d} \partial_{y_j}  \rho(t,y) \ dy + C \int_{|x-y| > R_1} \frac{x_i - y_i}{|x - y|^d} \partial_{y_j}  \rho(t,y) \ dy\\
& = & C \int_{|x-y| < R_1} \frac{x_i - y_i}{|x - y|^d} \partial_{y_j}  \rho(t,y) \ dy + C \int_{|x-y| = R_1} \frac{x_i - y_i}{|x - y|^d}  \rho(t,y) \frac{x_j-y_j}{|x-y|}\ dS_y\\
& \ & + C \int_{R_1 < |x-y| < R_2} \partial_{y_j} \left ( \frac{x_i - y_i}{|x - y|^d} \right )  \rho(t,y) \ dy + C \int_{|x-y| > R_2} \partial_{y_j} \left ( \frac{x_i - y_i}{|x - y|^d} \right ) \rho(t,y) \ dy\\
& =: & I - IV.
\end{eqnarray*}
Next, we estimate each contribution so that
$$I \lesssim \| \nabla_x \rho(t) \|_\infty   \left ( \int_{|x-y| < R_1} |x-y|^{1-d} \ dy \right ) \lesssim \| \nabla_x \rho(t) \|_\infty R_1,$$
and
$$II \lesssim \|\rho(t) \|_\infty \left ( \int_{|x-y| = R_1} |x-y|^{1-d} \ dS_y \right ) \lesssim \|\rho(t) \|_\infty.$$
To estimate the final two terms, we use
$$\left | \partial_{y_j} \left ( \frac{x_i - y_i}{|x - y|^d} \right ) \right | \leq C|x-y|^{-d}$$
to find
$$III \lesssim \|\rho(t) \|_\infty \left ( \int_{R_1 < |x-y| < R_2} |x-y|^{-d} \ dy \right ) \lesssim \|\rho(t) \|_\infty \ln^* \left ( \frac{R_2}{R_1} \right ),$$ 
and
$$IV \lesssim R_2^{-d} \| \rho(t) \|_1 \lesssim R_2^{-d}$$
where
$$\ln^*(s) =
\begin{cases}
0, & \mathrm{if} \ s \leq 1\\
\ln(s), & \mathrm{if} \ s \geq 1.
\end{cases}
$$
Taking 
$$R_1 = \frac{ \| \rho(t) \|_\infty}{\|\nabla_x \rho(t) \|_\infty} \qquad \mathrm{and} \qquad R_2 = \| \rho(t) \|_\infty^{-\frac{1}{d}}$$
yields
\begin{equation}
\label{gradEstd}
\|\nabla_x E(t) \|_\infty \lesssim \left ( 1 + \ln^* \left (\frac{\|\nabla_x \rho(t)\|_\infty}{\|\rho(t) \|_\infty^\frac{d+1}{d}} \right ) \right ) \| \rho(t) \|_\infty.
\end{equation}
We note that this bound is increasing in the contribution of $\|\rho(t)\|_\infty$, and using Lemma \ref{momrho} in \eqref{gradEstd} yields
$$
\|\nabla_x E(t) \|_\infty \lesssim \left ( 1 + \ln^* \left ((t+\alpha) ^{d+1} \|\nabla_x \rho(t)\|_\infty  \right ) \right )(t+\alpha)^{-d}.
$$
As in the proof of Lemma \ref{momrho}, we can bound the derivative of $\rho$ using moments of derivatives of $g$ so that
$$\left |\nabla_x \rho(t,x) \right | \lesssim (t+ \alpha)^{-d} \int \left |\nabla_x g \left (t,w,\frac{x-w}{t+\alpha} \right ) \right | \ dw \lesssim (t+ \alpha)^{-d} \mcGx(t),$$
and thus
\begin{equation}
\label{DEdecayln}
\|\nabla_x E(t) \|_\infty \lesssim \biggl ( 1 + \ln \left ( \mcGx(t) \right ) \biggr) (t+\alpha)^{-q}
\end{equation}
for any $0<q<d$ that can be chosen as close to $d$ as desired.
%

Next, we estimate derivatives of $g$ in order to close the argument.
Denoting the translated Vlasov operator by $\mfV$ as before so that
$$\mfV h = \partial_t h - (t+\alpha) E(t, x+ v(t+\alpha)) \cdot \nabla_x h + E(t, x+ v(t+\alpha)) \cdot \nabla_v h,$$
we take derivatives in the Vlasov equation and apply $\langle x \rangle^p$ to find
$$\mfV \biggl (\langle x \rangle^p \partial_{v_k} g \biggr ) = - p (t+\alpha) \langle x \rangle^{p-2}  x \cdot E(t, x+ v(t+\alpha)) \partial_{v_k} g(t,x,v) + (t+\alpha)\left ( (t+\alpha)  \nabla_x g - \nabla_v g \right ) \cdot \partial_{x_k}E(t, x+ v(t+\alpha)).$$
Inverting the operator by integrating along characteristics and taking supremums then gives
\begin{eqnarray*}
\| \langle x \rangle^p \partial_{v_k} g(t)\|_\infty & \leq & \| \langle x \rangle^p \partial_{v_k} g(0)\|_\infty +C \int_0^t (s+\alpha) \| E(s)\|_\infty  \|\langle x \rangle^{p} \nabla_vg(s)\|_\infty\ ds\\
& \ & + \int_0^t (s+ \alpha) \|\nabla_x E(s) \|_\infty \biggl ( (s+\alpha)\|\nabla_x g(s) \|_\infty + \|\nabla_v g(s)\|_\infty \biggr ) \ ds\\
& \leq & \mcGv(0)  +C \int_0^t (s+\alpha)^{2-d}  \mcGv(s) ds\\
& \ & + \int_0^t (s+ \alpha)^{1-q}\biggl ( 1 + \ln \left ( \mcGx(s) \right ) \biggr) \biggl ( (s+\alpha) \mcGx(s) + \mcGv(s) \biggr ) \ ds
\end{eqnarray*}
where we have used \eqref{DEdecayln} to estimate field derivatives.
Summing over $k=1,...,d$ gives
\begin{equation}
\label{Gv1}
\mcGv(t) \lesssim 1 +  \int_0^t (s+ \alpha)^{2-q}\biggl ( 1 + \ln \left ( \mcGx(s) \right ) \biggr) \biggl (\mcGx(s) + \mcGv(s) \biggr ) \ ds.
\end{equation}

We estimate $x$ derivatives in the same manner to find
$$\mfV \biggl (\langle x \rangle^p \partial_{x_k} g \biggr ) = - p (t+\alpha) \langle x \rangle^{p-2}  x \cdot E(t, x+ v(t+\alpha)) \partial_{x_k} g(t,x,v) + \left ( (t+\alpha)  \nabla_x g - \nabla_v g \right ) \cdot \partial_{x_k}E(t, x+ v(t+\alpha)),$$
and thus
\begin{equation}
\label{Gx1}
\mcGx(t) \lesssim 1 +  \int_0^t (s+ \alpha)^{1-q}\biggl ( 1 + \ln \left ( \mcGx(s) \right ) \biggr) \biggl (\mcGx(s) + \mcGv(s) \biggr ) \ ds.
\end{equation}
Defining
$$\mcD(t) = e^2 + \mcGx(t) + \mcGv(t)$$
and adding \eqref{Gv1} and \eqref{Gx1} yields
$$\mcD(t) \lesssim 1 +  \int_0^t (s+ \alpha)^{2-q} \mcD(s)\ln \left (\mcD(s) \right )\ ds.$$
Invoking a variant of Gronwall's inequality then yields
$$\mcD(t) \lesssim \exp \left ( \exp \left ( \int_0^t (s+\alpha)^{2-q} \ ds \right ) \right ) \lesssim 1$$
as $q$ is sufficiently close to $d \geq 4$. 
As $\mcD(t)$ is bounded, we find
$$\mcGv(t) \lesssim 1 \qquad \mathrm{and} \qquad \mcGx(t) \lesssim 1,$$
and the second and third conclusions follow.
Additionally, the first conclusion is obtained upon using the bound on $\mcGx(t)$ within the estimate of $\|\nabla_x E(t) \|_\infty$ given by \eqref{DEdecayln}.
\end{proof}

%

\subsection{Convergence of the spatial average}
\label{sect:vellim}

Because the field and charge density decay rapidly in time and velocity derivatives of $g$ are uniformly bounded, we can establish the convergence of spatial averages.
\begin{lemma}
\label{Funif}
There exists a continuous $F_\infty \in L^1(\bfR^d) \cap L^\infty(\bfR^d)$ such that
$$F(t,v) = \int f(t,x, v) \ dx$$
satisfies $F(t,v) \to F_\infty(v)$ in $C(\bfR^d)$ as $t \to \infty$ with
$$\| F(t) - F_\infty \|_\infty \lesssim (t+\alpha)^{2-d}.$$
\end{lemma}
\begin{proof}
Upon integrating the Vlasov equation of \eqref{VPg} in $x$ and integrating by parts, we find 
\begin{eqnarray*}
\left | \partial_t \int g(t,x,v) \ dx \right | & = & \left |  \int E(t, x+v(t+\alpha)) \cdot ((t+\alpha) \nabla_x - \nabla_v) g(t,x,v) \ dx \right |\\
& = & \left |  (t+\alpha)  \int \rho(t, x+v(t+\alpha)) g(t,x,v) \ dx +  \int E(t, x+v(t+\alpha)) \cdot \nabla_v g(t,x,v) \ dx \right |\\
& \lesssim & (t+\alpha)  \Vert \rho(t) \Vert_\infty F(t,v) + \Vert E(t) \Vert_\infty \mcG_v(t).
\end{eqnarray*}
Thus, we use Lemmas \ref{momrho} and \ref{DEf} to find
$$\left | \partial_t F(t,v) \right | \lesssim (t+\alpha)^{1-d} F(t,v) + (t+\alpha)^{1-d}.$$
As 
$\|F(0)\|_\infty  \leq \mcG(0)$
and the latter term above is integrable in time, we find
$$F(t,v) \leq F(0,v) + \int_0^t \left | \partial_t F(s,v) \right | \ ds \lesssim 1 + \int_0^t (s+\alpha)^{1-d} F(s, v) \ ds,$$
and after taking the supremum and invoking Gronwall's inequality, this yields
\begin{equation}
\label{FLinfty}
\| F(t) \|_\infty \lesssim \exp \left (\int_0^t (s+\alpha)^{1-d} \ ds \right ) \lesssim 1.
\end{equation}
Returning to the estimate of $\partial_t F$, we use the uniform bound on $\|F(t)\|_\infty$ to find
$$\left | \partial_t F(t,v) \right |  \lesssim (t+\alpha)^{1-d},$$
which implies that $ \| \partial_t F(t)\|_\infty$ is integrable.
This bound then establishes the estimate for $s \geq t$
$$\Vert F(t) - F(s) \Vert_\infty = \left \Vert \int_s^t \partial_t F(\tau) \ d\tau \right \Vert_\infty
\leq \int_t^s \Vert \partial_t F(\tau) \Vert_\infty \ d\tau
 \lesssim (t+\alpha)^{2-d},$$
and taking $s \to \infty$ establishes the limit. 
More precisely, as $F(t,v)$ is continuous and the limit is uniform, there is $F_\infty \in C(\bfR^d)$ such that
$$ \| F(t) - F_\infty \|_\infty  \lesssim (t+\alpha)^{2-d}.$$
Furthermore, as $\int F(t,v) \ dv  = \mcM$ for every $t \geq 0$, we have $F_\infty \in L^1(\bfR^d)$ with
$ 0 \leq \int F_\infty(v) \ dv \leq \mcM$.

\end{proof}

\subsection{Convergence of the field and macroscopic densities}
\label{sect:fieldconv}

Now that we have shown the convergence of $F(t,v)$, we establish the precise asymptotic profile of the field and the charge and current densities.
From the limiting density $F_\infty(v)$, we define its induced electric field by
$$E_\infty(v) = \nabla_v (\Delta_v)^{-1} F_\infty(v) = \frac{1}{d\omega_d} \int \frac{\xi}{|\xi|^d} F_\infty(v -\xi) \ d\xi$$
for every $v \in \bfR^d$.
To ensure the necessary regularity of the limiting field we note that due to Lemma \ref{LE} $\| E_\infty\|_\infty < \infty$, as $F_\infty \in L^1(\bfR^d) \cap L^\infty(\bfR^d)$ by Lemma \ref{Funif}.
With this, we establish a refined estimate of the electric field.

\begin{lemma}
\label{LField}
We have
$$\sup_{x \in \bfR^d} \left | (t+\alpha)^{d-1} E(t,x) - E_\infty\left (\frac{x}{t+\alpha} \right) \right | \lesssim (t+\alpha)^\frac{1-d}{d} $$
\end{lemma}
\begin{proof}
In order to properly decompose the difference of these quantities, we first represent the field in terms of the translated distribution function. In particular, we have
$$E(t,x) =  \frac{1}{d\omega_d}  \iint \frac{x- y}{\left | x - y \right |^d} \ g(t, y - u(t+\alpha), u) \ du dy,$$
which, upon performing the change of variables
$$\xi = \frac{x- y}{t+\alpha}$$
with respect to $y$
and
$$w = x - (u+\xi)(t+\alpha)$$
with respect to $u$, gives
$$E(t,x) = (t+\alpha)^{1-d}   \frac{1}{d\omega_d} \iint \frac{\xi}{|\xi|^d} \ g \left (t, w, \frac{x-w}{t+\alpha} - \xi \right ) \ dwd\xi.$$
Therefore, due to the convolution structure of $E_\infty$ we have
\begin{equation}
\label{Dform}
(t+\alpha)^{d-1} E(t,x) - E_\infty\left (\frac{x}{t+\alpha} \right) =   \frac{1}{d\omega_d} \int \frac{\xi}{|\xi|^d}  \left ( \int g \left (t, w, \frac{x-w}{t+\alpha} - \xi \right )  dw- F_\infty \left (\frac{x}{t+\alpha} - \xi \right) \right )\ d\xi.
\end{equation}
Next, we split the $\xi$-integrand so that
$$ \int g \left (t, w, \frac{x-w}{t+\alpha} - \xi \right ) \ dw - F_\infty \left ( \frac{x}{t+\alpha} - \xi \right) = \mcA_1\left (t,\frac{x}{t+\alpha}-\xi \right) + \mcA_2\left (t,\frac{x}{t+\alpha}-\xi \right)$$
where
\begin{equation}
\label{G1}
\mcA_1(t,v) =    \left (\int g(t,w, v) dw - F_\infty(v) \right ) = F(t,v) - F_\infty(v)
\end{equation}
and
\begin{equation}
\label{G2}
\mcA_2(t,v) =  \int \left ( g \left (t, w, v  - \frac{w}{t+\alpha} \right ) -g(t,w, v) \right )dw.
\end{equation}
Using this decomposition in \eqref{Dform}, we have
$$\left |(t+\alpha)^{d-1} E(t,x) - E_\infty\left (\frac{x}{t+\alpha} \right) \right | \leq  \left \| \nabla_v (\Delta_v)^{-1} \mcA_1(t) \right \|_\infty + \left \| \nabla_v (\Delta_v)^{-1} \mcA_2(t) \right \|_\infty.$$
To estimate the convolution terms on the right side of the inequality, we will use Lemma \ref{LE}.

Now, to estimate the $\mcA_1$ term we find
$$\|\mcA_1(t) \|_1 = \| F(t) - F_\infty \|_1 \leq 2\mcM \lesssim 1$$
and of course
$$\|\mcA_1(t) \|_\infty =  \| F(t) - F_\infty \|_\infty \lesssim (t+\alpha)^{2-d}.$$
Using these estimates with Lemma \ref{LE} yields
\begin{equation}
\label{I}
\| \nabla_v (\Delta_v)^{-1} \mcA_1(t) \|_\infty \lesssim \| F(t) - F_\infty \|_\infty \lesssim (t+\alpha)^{\frac{(d-1)(2-d)}{d}}.
\end{equation}

To control the $\mcA_2$ term, we use the bound on spatial moments and velocity derivatives of $g$, which yields
\begin{eqnarray*}
\| \mcA_2(t) \|_\infty 
& = & \sup_{v \in \bfR^d} \left |   \int \left [g \left (t, w, v  - \frac{w}{t+\alpha} \right ) -  g(t,w, v) \right ] \ dw \right |\\
& \lesssim & \sup_{v \in \bfR^d}  \int \left |\int_0^1 \frac{d}{d\theta} \left (g \left (t, w, v - \theta \frac{w}{t+\alpha} \right )  \right )  d\theta \right |  dw\\
& \lesssim & (t+\alpha)^{-1} \sup_{v \in \bfR^d}  \int_0^1 \int |w| \left |\nabla_v g \left (t, w, v - \theta \frac{w}{t+\alpha} \right ) \right |  dw d\theta\\
& \lesssim & (t+\alpha)^{-1}  \mcGv(t) \left (\int \langle w \rangle^{1-p}  dw \right )\\
& \lesssim & (t+\alpha)^{-1}.
\end{eqnarray*}
In order to estimate $\| \mcA_2(t,x) \|_1$, we merely use charge conservation so that
$$\| \mcA_2(t) \|_1 \leq  \iint \left | g \left (t, w, v - \frac{w}{t+\alpha} \right ) - g(t,w, v) \right | dw dv \leq 2\mcM \lesssim 1.$$
Combining this with the bound on $\| \mcA_2(t) \|_\infty$ within Lemma \ref{LE} gives
\begin{equation}
\label{II}
\| \nabla_v (\Delta_v)^{-1} \mcA_2(t) \|_\infty \lesssim (t+\alpha)^\frac{1-d}{d}.
\end{equation}
Finally, collecting \eqref{I} with $d \geq 4$ and \eqref{II}
we conclude
$$\sup_{x \in \bfR^d} \left |(t+\alpha)^{d-1} E(t,x) - E_\infty\left (\frac{x}{t+\alpha} \right) \right | \lesssim (t+\alpha)^\frac{1-d}{d}.$$ 
\end{proof}

Next, we estimate the charge density using the same tools.

\begin{lemma}
\label{LDensity}
We have
$$\sup_{x \in \bfR^d} \left | (t+\alpha)^d \rho(t,x) - F_\infty \left (\frac{x}{t + \alpha} \right) \right | \lesssim (t+\alpha)^{-1}.$$
\end{lemma}
\begin{proof}
As for the field, we must rewrite this difference in terms of the translated distribution function. To this end, we have
$$\rho(t,x) =    \int g (t, x - u(t+\alpha), u) \ du,$$
and, upon performing the change of variables
$$y = x-u(t+\alpha)$$
with respect to $u$, we find 
$$\rho(t,x) = (t+\alpha)^{-d}    \int g \left (t, y, \frac{x-y}{t+\alpha} \right ) \ dy.$$
Hence, the difference of the densities can be split into two terms as
\begin{eqnarray*}
\left | (t+\alpha)^d \rho(t,x) - F_\infty \left (\frac{x}{t+\alpha} \right) \right | &\leq & \left |    \int \left [g \left (t, y, \frac{x-y}{t+\alpha} \right ) - g\left (t, y, \frac{x}{t+\alpha} \right ) \right]\ dy \right | \\ 
& \ &  + \left | F \left (t, \frac{x}{t+\alpha} \right) - F_\infty \left (\frac{x}{t+\alpha} \right ) \right|\\
& =: & I + II.
\end{eqnarray*}

Using methods similar to the previous lemma, the first term satisfies
$$I \lesssim (t+\alpha)^{-1} \int_0^1 \int |y| \left | \nabla_v g\left (t, y, \frac{x-\theta y}{t+\alpha} \right ) \right | dy d\theta \lesssim (t+\alpha)^{-1} \mcGv(t) \left (\int \langle y \rangle^{1-p} \ dy \right ) \lesssim (t+\alpha)^{-1},$$
while the second term is straightforward, namely
$$II \leq 
\| F(t) - F_\infty \|_\infty \lesssim (t+\alpha)^{2-d}.$$
Combining these estimates and using $d \geq 4$ then yields the stated result.
\end{proof}

Finally, we estimate the current density in a similar fashion, but restricted to spatial subsets with linear growth in $t$.
\begin{lemma}
\label{LCurrent}
We have
$$\sup_{|x| \lesssim t+\alpha} \left | (t+\alpha)^d j(t,x) - \frac{x}{t+\alpha} F_\infty\left ( \frac{x}{t+\alpha} \right) \right | \lesssim (t+\alpha)^{-1}.$$
\end{lemma}
\begin{proof}
Throughout, we consider only $|x| \lesssim t + \alpha$.
Performing the same change of variables as in Lemma \ref{LDensity} transforms $j(t,x)$ into
$$j(t,x) = (t+\alpha)^{-d}    \int \left ( \frac{x-y}{t+\alpha} \right ) g \left (t, y, \frac{x-y}{t+\alpha} \right ) \ dy.$$
Hence, the difference can be split into three terms as
\begin{eqnarray*}
\left | (t+\alpha)^d j(t,x) - \frac{x}{t+\alpha}F_\infty\left ( \frac{x}{t+\alpha} \right) \right | &\leq & \left |   \int \frac{y}{t+\alpha}g \left (t, y, \frac{x-y}{t+\alpha} \right ) \ dy \right |\\
& \ & +  \left |    \int \frac{x}{t+\alpha} \left [ g \left (t, y, \frac{x-y}{t+\alpha} \right )  - g \left ( t, y, \frac{x}{t+\alpha} \right) \right ] \ dy\right | \\
& \ & + \left | \frac{x}{t+\alpha} \left (F \left (t,\frac{x}{t+\alpha} \right) - F_\infty \left (\frac{x}{t+\alpha} \right) \right ) \right|\\
& =: & I + II + III.
\end{eqnarray*}

The first term is estimated using $\mcG(t)$ so that
$$I \lesssim (t+\alpha)^{-1} \mcG(t)  \left (\int \langle y \rangle^{1-p} \ dy \right ) \lesssim (t+\alpha)^{-1} .$$
The second term has similar structure, but involves velocity derivatives of $g$, and we find
$$II\lesssim |x| (t+\alpha)^{-2}  \mcGv(t) \left ( \int \langle y \rangle^{1-p} dy \right ) \lesssim (t+\alpha)^{-1}.$$
Finally, the third term is straightforward and yields
$$III \lesssim |x| (t+\alpha)^{-1} \| F(t) - F_\infty \|_\infty \lesssim (t+\alpha)^{2-d} .$$
Combining these estimates then yields the stated result.
\end{proof}

\subsection{Scattering of the distribution function}
\label{sect:modscatt}

With the field and derivative estimates solidified, we prove that the distribution functions scatter to a limiting value as $t \to \infty$ along their free-streaming trajectories in phase space.

\begin{lemma}
\label{LScattering}
There exists a continuous $f_\infty \in L^1(\bfR^{2d}) \cap L^\infty(\bfR^{2d})$ such that
$$g(t,x,v) \to f_\infty(x,v)$$ 
uniformly as $t \to \infty$.
In particular, we have
$$\sup_{(x,v) \in \bfR^{2d}} \left |f \left (t,x + v(t+\alpha), v \right ) - f_\infty(x,v) \right |  \lesssim (t+\alpha)^{3-d}.$$
\end{lemma}

\begin{proof}
Because $g$ satisfies \eqref{VPg}, we have
$$\partial_t g =  (t+\alpha) E(t, x+v(t+\alpha)) \cdot \nabla_x g(t,x,v) -   E(t, x+v(t+\alpha))\cdot \nabla_v g(t,x,v).$$
Similar to the proof of Lemma \ref{Funif}, we wish to show that 
$\displaystyle \| \partial_t g(t) \|_\infty$ is integrable in order to establish the existence of a limiting function in this norm.

To this end, we merely estimate the terms on the right side of the equation.
Using the field decay and Lemma \ref{DEf}, we find
$$\left |(t+\alpha) E(t, x+v(t+\alpha)) \nabla_x g(t,x,v) \right | \lesssim (t+\alpha) \|E(t) \|_\infty \mcGx(t) \lesssim (t+\alpha)^{2-d}$$
and
$$\left | E(t, x+v(t+\alpha)) \nabla_v g(t,x,v) \right | \lesssim \|E(t) \|_\infty \mcGv(t) \lesssim (t+\alpha)^{1-d}.$$
Combining yields
$$ \|\partial_t g(t) \|_\infty \lesssim (t+\alpha)^{2-d}.$$
As $d \geq 4$, this bound is integrable in time and there is $f_\infty \in C(\bfR^{2d})$ such that
$$\| g(t) - f_\infty \|_\infty \lesssim (t+\alpha)^{3-d}.$$
Similar to the spatial average, the limiting function is integrable as $$ \iint g(t,x,v) \ dv dx = \mcM $$
for all $ t\geq 0$, and bounded as $\|g(t) \|_\infty \leq \| f_0 \|_\infty$ for all $ t\geq 0$.

\end{proof}



With these lemmas firmly in place, Theorem \ref{T1} follows by merely collecting the stated estimates.

\section{Small moment solutions and proof of Theorem \ref{T3}}
\label{sect:T3}

In the final section, we establish the global-in-time existence of small data solutions.
We note that the smallness condition neither involves derivatives of $f_0$ nor restricts the value of $\|f_0\|_\infty$.

\begin{proof}[Proof of Theorem \ref{T3}]

Let $M_n(0) = \epsilon > 0$ for some $n > d(d-1)$ and denote the maximal time of existence by $T_{\mathrm{max}} > 0$. We will impose conditions on $\epsilon$ as we continue.
Let $$T_\infty = \sup \{ t \geq 0 : M_n(t) \leq 2\epsilon \}.$$
Notice that $T_\infty > 0$ by continuity.
Throughout, constants may depend upon $d$, $\alpha$, and $n$, but not on any other quantities.

Next, we employ Corollary \ref{EM} with $k=n$ to find
$$\| E(t) \|_\infty \leq C(t+\alpha)^{1-d} \ M_n(t)^\frac{(d-2)(d+1)}{d(n-2)} \leq C(t+\alpha)^{1-d} \ \epsilon^\frac{(d-2)(d+1)}{d(n-2)} $$
for all $t \in [0,T_\infty)$.
In particular, though the corollary (and Lemmas \ref{LE} and \ref{Lrho}, on which it depends) only state estimates for $t$ sufficiently large, the proofs hold for any $t \geq 0$ whenever $\alpha > 0$.
We use this field estimate along with Lemma \ref{ME} for $k=n$ so that
\begin{eqnarray*}
M_n(t)^\frac{1}{n-2} & \leq & M_n(0)^\frac{1}{n-2} + C\int_0^t (s+ \alpha) \| E(s) \|_\infty \ ds\\
& \leq & \epsilon^\frac{1}{n-2} + C\epsilon^\frac{(d-2)(d+1)}{d(n-2)} \int_0^\infty (s + \alpha)^{2-d} \ ds  \\
& \leq & \epsilon^\frac{1}{n-2} + C \left (\epsilon^\frac{(d-2)(d+1)}{d} \right )^\frac{1}{n-2}
\end{eqnarray*}
for $t \in [0,T_\infty)$.
For $d \geq 4$, we have
$$\frac{(d-2)(d+1)}{d} = \left ( 1- \frac{2}{d} \right ) (d+1) \geq \frac{5}{2}$$
and thus taking $\epsilon < 1$ implies
$$\left (\epsilon^\frac{(d-2)(d+1)}{d} \right )^\frac{1}{n-2} \leq \epsilon^\frac{2}{n-2}.$$
Therefore, we find
$$M_n(t)^\frac{1}{n-2} \leq  \epsilon^\frac{1}{n-2} + \left ( C  \epsilon^\frac{1}{n-2} \right ) \epsilon^\frac{1}{n-2} = \left ( 1 + C  \epsilon^\frac{1}{n-2} \right )  \epsilon^\frac{1}{n-2}, $$
and taking $\epsilon$ smaller if necessary yields
$$M_n(t) \leq \frac{3}{2} \epsilon < 2\epsilon$$
for all $t \in [0,T_\infty)$ and $\epsilon$ sufficiently small.
By the definition of $T_\infty$ it follows that $T_\infty = T_{\mathrm{max}}$, and this further yields
$$\|E(t) \|_\infty \leq C \epsilon^\frac{(d-2)(d+1)}{d(n-2)} (t+\alpha)^{1-d}$$
for the lifespan of the solution.
The field estimate then provides a uniform upper bound on moments of $g$, and $T_{\mathrm{max}} = \infty$ follows.
With the field decay established, Theorem \ref{T1} can be applied to provide the complete asymptotic behavior of solutions, as well.
\end{proof}



\end{document}